\documentclass[12pt]{article}

\usepackage{amsmath}
\usepackage{latexsym}
\usepackage{amssymb}
\usepackage{ifthen}
\usepackage{tikz}
\usepackage{amsthm}
\usepackage{yfonts}

\usepackage{lineno}

\usepackage{hyperref}

\usepackage{enumitem}

\newcommand{\junk}[1]{}

\newcommand{\formal}[1]{\ensuremath{\textsf{#1}}}

\newcommand{\dword}[1]{\textit{#1}}

\newcommand{\graph}[1]{\ensuremath{\mathit #1}}

\newcommand{\B}{\ensuremath{\mathcal{B}}}

\newcommand{\inside}{\ensuremath{\formal{in}}}
\newcommand{\outside}{\ensuremath{\formal{out}}}
\newcommand{\armin}[2]{\ensuremath{(#1)^{#2}_{\inside}}}
\newcommand{\armout}[2]{\ensuremath{(#1)^{#2}_{\outside}}}
\newcommand{\armvar}[3]{\ensuremath{(#1)^{#2}_{#3}}}

\newcommand{\edge}[2]{\ensuremath{ \{#1, #2\}  }}

\newcommand{\arm}{\graph{A}}

\newcommand{\setbar}{\ensuremath{:}}

\newtheorem{defn}{Definition}[section]
\newtheorem{thm}[defn]{Theorem}

\newtheorem{lemma}[defn]{Lemma}

\newtheorem{claim}[defn]{Claim}

\newtheorem*{lemmaNoNum}{Lemma}
\newtheorem*{obs}{Observation}

\title {Comparing the power of cops to zombies in pursuit-evasion games}

\author{David Offner\\
\small Department of Mathematics and Computer Science\\[-0.8ex]
\small Westminster College\\[-0.8ex] 
\small New Wilmington, PA, U.S.A.\\
\small\tt offnerde@westminster.edu\\
\and
Kerry Ojakian\\
\small Department of Mathematics and Computer Science\\[-0.8ex]
\small City University of New York \\[-0.8ex]
\small Bronx, NY, U.S.A\\
\small\tt kerry.ojakian@bcc.cuny.edu
}

\begin{document}
\maketitle

\begin{abstract}  

We compare two kinds of pursuit-evasion games played on graphs.
In Cops and Robbers, the cops can move strategically to adjacent vertices as they please, while in a new variant, called deterministic Zombies and Survivors, the zombies (the counterpart of the cops) are required to always move towards the survivor (the counterpart of the robber).
The cop number of a graph is the minimum number of cops required to catch the robber on that graph; the zombie number of a graph is the minimum number of zombies required to catch the survivor on that graph. We answer two questions from the 2016 paper of Fitzpatrick, Howell, Messinger, and Pike.  We show that for any $m \ge k \ge 1$, there is a graph with zombie number $m$ and cop number $k$.  We also show that the zombie number of the 
$n$-dimensional hypercube is $\lceil 2n/3\rceil$.

\end{abstract}


\section{Introduction}

The game of Cops and Robbers is a perfect-information two-player pursuit-evasion game played on a  graph.  One player controls a group of cops and the other player controls a single robber. To begin the game, the cops and robber each choose vertices to occupy, with the cops choosing first.  Play then alternates between the cops and the robber, with the cops moving first. On a turn a player may move to an adjacent vertex or stay still.  If any cop and the robber ever occupy the same vertex, the robber is caught and the cops win.  On a graph \graph{G}, the fewest number of cops required to catch the robber is called the \dword{cop number} and denoted
$c(\graph{G})$.  If $c(\graph{G})=k$ we say that \graph{G} is $k$-cop-win. 
 The game was introduced by Nowakowski and Winkler \cite{NW83}, and Quilliot \cite{Qui78}. A nice introduction to the game and its many variants is found in 
the book by Bonato and Nowakowski \cite{BN11}. 

In this paper we discuss a variation of Cops and Robbers introduced by Fitzpatrick, Howell, Messinger, and Pike \cite{FHMP16} called deterministic Zombies and Survivors;
a probabilistic version was introduced by  Bonato, Mitsche, Perez-Gimenez, and Pralat in \cite{BMGP_ProbZombies_2016}.
In deterministic Zombies and Survivors, the cops are replaced by zombies and the robber is replaced by a survivor.
The rules of this variant are the same as Cops and Robbers, except for the following significant restriction on the zombies: each zombie must move on every turn, and furthermore, must move closer to the survivor on every turn. If there are multiple moves available to the zombies, they can coordinate and choose their moves intelligently. The number of zombies required to capture the survivor on a graph \graph{G} is known as the \dword{zombie number} and denoted $z(\graph{G})$.

We focus on a comparison between the zombie number and the cop number.
Since any zombie strategy can be played by the same number of cops, for all graphs \graph{G}, 
$z(\graph{G}) \ge c(\graph{G})$. Fitzpatrick, Howell, Messinger, and Pike \cite{FHMP16} 
gave examples of graphs where $z(\graph{G}) > c(\graph{G})$ and asked about the relationship between these two parameters. Using hypercube graphs, they noted that the gap $z(\graph{G})-c(\graph{G})$ can be arbitrarily large.  They also observed that the graph $\graph{G}_5$ (See Figure~\ref{G5}) is 1-cop-win but requires 2 zombies to win, 
and so $z(\graph{G})/c(\graph{G})$ can be at least 2.  In Question~19, they asked how large the ratio $z(\graph{G})/c(\graph{G})$ can be.   We show that any ratio of 1 or larger is possible, as in Section~\ref{finalg}, we describe a family of graphs $\graph{Z}_{k,m}$ such that for any integers $m \ge k \ge 1$, $z(\graph{Z}_{k,m}) =m$ and $c(\graph{Z}_{k,m}) = k$. 
These graphs are constructed by combining a so-called base graph that has cop and zombie number $k$ with ``arms'' on which 1 cop can catch a robber, but many zombies may be required to catch a survivor.  Thus $k$ cops can win by either catching the robber on the base graph, or forcing the robber onto one of the arms, and then catching the robber on the arm.  However while $k$ zombies can force the survivor onto one arm of the graph, it requires more zombies to then catch the survivor.  We describe the construction of the arm graphs in Section~\ref{constr}, and then describe the base graph and final construction in Section~\ref{zkmsec}. 

In Section~\ref{cube} we prove that Conjecture~18 from \cite{FHMP16} is true, that is, the zombie number for the $n$-dimensional hypercube is $\lceil 2n/3\rceil$.
This conjecture is also proved by Fitzpatrick in a different manner in
\cite{FitzpatrickArxiv2018} as a corollary of a more general result.
Note that the cop number of the $n$-dimensional hypercube is 
$\lceil (n+1)/2 \rceil$, as shown by Maamoun and Meyniel in \cite{MM87}. 

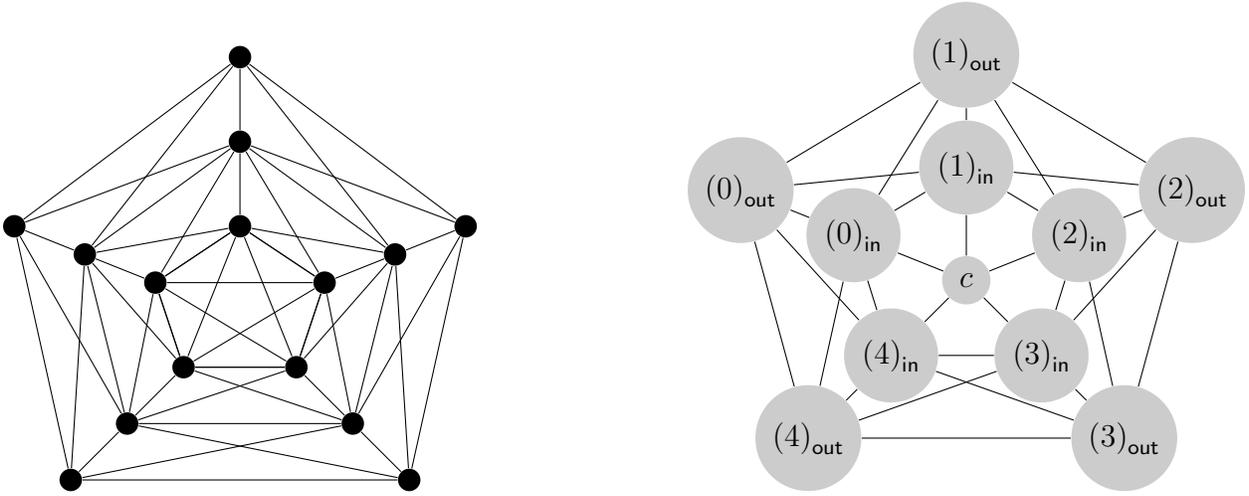
\begin{figure}
\begin{center}
  \begin{tikzpicture}[every node/.style={circle, inner sep=3pt, fill=black}, scale=.75]
      \node (11) at (-1,-1) {};
      \node (21) at (-2,-2) {};
      \node (31) at (-3,-3) {};
      \node (12) at (1,-1) {};
      \node (22) at (2,-2) {};
      \node (32) at (3,-3) {};
      \node (13) at (1.5,.5) {};
      \node (23) at (2.75,1) {};
      \node (33) at (4,1.5) {};
      \node (14) at (0,1.5) {};
      \node (24) at (0,3) {};
      \node (34) at (0,4.5) {};
      \node (10) at (-1.5,.5) {};
      \node (20) at (-2.75,1) {};
      \node (30) at (-4,1.5) {};
    \foreach \from/\to in {11/21, 21/31, 12/22, 22/32, 13/23, 23/33, 14/24, 24/34, 10/20, 20/30}
    \draw (\from) -- (\to);
    \foreach \from/\to in {11/12, 11/13, 11/14, 11/10, 12/13, 12/14, 12/10, 13/14, 13/10, 14/10}
    \draw (\from) -- (\to);
    \foreach \from/\to in {11/12, 12/13, 13/14, 14/10, 10/11, 20/21, 21/22, 22/23, 23/24, 24/20, 31/32, 32/33, 33/34, 34/30, 30/31}
    \draw (\from) -- (\to);
        \foreach \from/\to in {11/22, 12/23, 13/24, 14/20, 10/21, 20/31, 21/32, 22/33, 23/34, 24/30, 31/22, 32/23, 33/24, 34/20, 30/21, 21/12, 22/13, 23/14, 24/10, 20/11}
    \draw (\from) -- (\to);

    \end{tikzpicture}
    \hfill
  \begin{tikzpicture}[every node/.style={circle, fill=black!20}]
      \node (11) at (-1,-1) {\armin{4}{}}; 
      \node (21) at (-2.1,-2.1) {\armout{4}{}}; 
      \node (12) at (1,-1) {\armin{3}{}}; 
      \node (22) at (2.1,-2.1) {\armout{3}{}}; 
      \node (13) at (1.5,.6) {\armin{2}{}}; 
      \node (23) at (3,1.2) {\armout{2}{}}; 
      \node (14) at (0,1.5) {\armin{1}{}}; 
      \node (24) at (0,3) {\armout{1}{}}; 
      \node (10) at (-1.5,.6) {\armin{0}{}}; 
      \node (20) at (-3,1.2) {\armout{0}{}}; 
      \node (c) at (0,0) {$c$};
    \foreach \from/\to in {11/21, 12/22,  13/23, 14/24, 10/20}
    \draw (\from) -- (\to);
\foreach \from/\to in {11/c,  12/c, 13/c, 14/c, 10/c}
    \draw (\from) -- (\to);
    \foreach \from/\to in {11/12, 12/13, 13/14, 14/10, 10/11, 20/21, 21/22, 22/23, 23/24, 24/20}
    \draw (\from) -- (\to);
        \foreach \from/\to in {11/22, 12/23, 13/24, 14/20, 10/21, 21/12, 22/13, 23/14, 24/10, 20/11}
    \draw (\from) -- (\to);
    \end{tikzpicture}
\end{center}
\caption{The graphs $\graph{G}_5$ (left) and \graph{T} (right)}
\label{G5}
\end{figure}

\section{The key building block: The arm graph}\label{constr}

In this section we define what we call an \dword{arm graph}, constructed using copies of a graph we call \graph{T}.  Then we prove two key properties about the arm graphs in Lemmas~\ref{thm_one_arm} and \ref{1armUB}.  In both lemmas we describe the same kind of graph: some graph with an arm graph attached to it. The first lemma will describe a scenario where the survivor has a winning strategy, while the second lemma will describe a scenario where the zombies have a winning strategy.

For a graph \graph{G} we denote its vertex set by $V(\graph{G})$, its edge set by $E(\graph{G})$,
and an edge between $x$ and $y$ is denoted \edge{x}{y}.
The graph  \graph{T} is shown on the right in Figure ~\ref{G5}.
Since a typical arm graph will contain multiple copies of $T$, we formally define the graph $\graph{T}_x$, isomorphic to \graph{T}, as follows, where the parameter $x$ is an
integer. 
\begin{align*} 
V(\graph{T}_x)&= \{ \ \armin{i}{x} \ \setbar \ 0 \le i \le 4 \ \} \ \cup \
 \{ \ \armout{i}{x} \ \setbar \ 0 \le i \le 4 \ \}
\ \cup \ \{c^x\}\\
E(\graph{T}_x)&= \{ \ \edge{ \armvar{i}{x}{a} }{ \armvar{j}{x}{b} } \ \setbar \ |i-j| = 1 \hbox{ or } 4\ \} \
\cup \ \{\ \edge{ \armin{i}{x} }{ \armout{i}{x} }, \ \edge{\armin{i}{x}}{c^x} \ \setbar \ 0 \le i \le 4 \ \} 
\end{align*}
Note that in the above definition,the variables $i$ and $j$ range over the integers from 0 to 4, inclusive, while the variables $a$ and $b$
range over the formal symbols \inside \ and \outside, whose names are meant to remind us, respectively, of ``inside'' vertices and ``outside'' vertices.
The graph \graph{T} is similar to other graphs appearing in the literature, such as the graphs $\graph{G}'$ and $\graph{G}_5$  from \cite{FHMP16} and the graph in the proof of Proposition 6 from \cite{BoyOBoy2011}.  

We now define the \dword{arm graphs} $A_n$ for any integer $n \ge 0$.  See Figure~\ref{ARM3} for a drawing of $\arm_3$. 
\begin{defn}
For $n \ge 0$, define $\arm_n$ to
consist of the vertices \armout{2}{-1} and \armout{0}{n}, along with graphs $\graph{T}_0, \graph{T}_1, \ldots, \graph{T}_{n-1}$, with the following edges added: $\{\edge{ \armout{2}{x} }{ \armout{0}{x+1} } \ \setbar \ -1 \le x \le n-1\}$.
\end{defn}
Note that $\arm_0$ is the graph with two vertices, \armout{2}{-1} and  \armout{0}{0}, and a single edge between these two vertices.

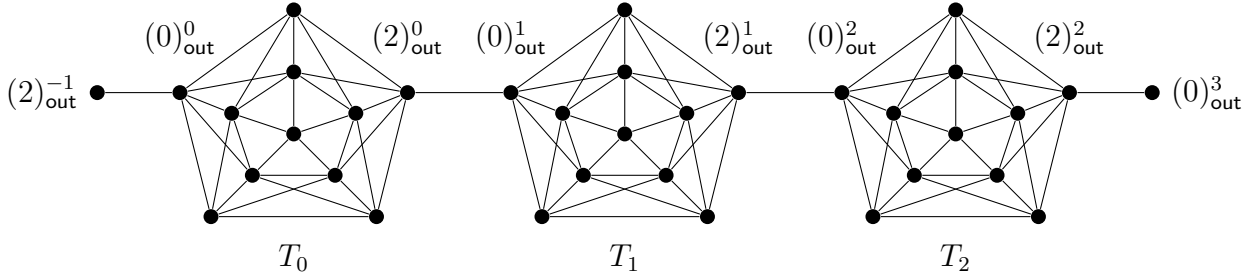
\begin{figure}
\begin{center}
  \begin{tikzpicture}[every node/.style={circle, inner sep=2pt, fill=black}, scale=.55]
      \node (11) at (-1,-1) {};
      \node (21) at (-2,-2) {};
      \node (12) at (1,-1) {};
      \node (22) at (2,-2) {};
      \node (13) at (1.5,.5) {};
      \node (23) at (2.75,1) [label=above: {\armout{2}{0}}]{};
      \node (14) at (0,1.5) {};
      \node (24) at (0,3) {};
      \node (10) at (-1.5,.5) {};
      \node (20) at (-2.75,1) [label=above: {\armout{0}{0}} ] {};
      \node (211) at (-4.75,1) [label=left: {\armout{2}{-1}} ] {};
      \node (c) at (0,0) {};
            \node[fill=none,rectangle] at (0,-3) {$T_0$};  
    \foreach \from/\to in {20/211, 11/21, 12/22, 13/23, 14/24, 10/20}
    \draw (\from) -- (\to);
    \foreach \from/\to in {11/c, 12/c, 13/c, 14/c, 10/c}
    \draw (\from) -- (\to);
    \foreach \from/\to in {11/12, 12/13, 13/14, 14/10, 10/11, 20/21, 21/22, 22/23, 23/24, 24/20}
    \draw (\from) -- (\to);
        \foreach \from/\to in {11/22, 12/23, 13/24, 14/20, 10/21, 21/12, 22/13, 23/14, 24/10, 20/11}
    \draw (\from) -- (\to);
      \node (11a) at (7,-1) {};
      \node (21a) at (6,-2) {};
      \node (12a) at (9,-1) {};
      \node (22a) at (10,-2) {};
      \node (13a) at (9.5,.5) {};
      \node (23a) at (10.75,1) [label=above: {\armout{2}{1}}]{};
      \node (14a) at (8,1.5) {};
      \node (24a) at (8,3) {};
      \node (10a) at (6.5,.5) {};
      \node (20a) at (5.25,1) [label=above: {\armout{0}{1}}] {};
      \node (ca) at (8,0) {};
      \node[fill=none,rectangle] at (8,-3) {$T_1$};  
    \foreach \from/\to in {11a/21a, 12a/22a, 13a/23a, 14a/24a, 10a/20a}
    \draw (\from) -- (\to);
    \foreach \from/\to in {11a/ca, 12a/ca, 13a/ca, 14a/ca, 10a/ca, 23/20a}
    \draw (\from) -- (\to);
    \foreach \from/\to in {11a/12a, 12a/13a, 13a/14a, 14a/10a, 10a/11a, 20a/21a, 21a/22a, 22a/23a, 23a/24a, 24a/20a}
    \draw (\from) -- (\to);
        \foreach \from/\to in {11a/22a, 12a/23a, 13a/24a, 14a/20a, 10a/21a, 21a/12a, 22a/13a, 23a/14a, 24a/10a, 20a/11a}
    \draw (\from) -- (\to);
      \node (11b) at (15,-1) {};
      \node (21b) at (14,-2) {};
      \node (12b) at (17,-1) {};
      \node (22b) at (18,-2) {};
      \node (13b) at (17.5,.5) {};
      \node (23b) at (18.75,1)  [label=above: {\armout{2}{2}}]{};
      \node (14b) at (16,1.5) {};
      \node (24b) at (16,3) {};
      \node (10b) at (14.5,.5) {};
      \node (20b) at (13.25,1) [label=above: {\armout{0}{2}}]{};
      \node (cb) at (16,0) {};
      \node (203) at (20.75,1) [label=right: {\armout{0}{3}} ] {};
      \node[fill=none,rectangle] at (16,-3) {$T_2$};  
    \foreach \from/\to in {11b/21b, 12b/22b, 13b/23b, 14b/24b, 10b/20b}
    \draw (\from) -- (\to);
    \foreach \from/\to in {11b/cb, 12b/cb, 13b/cb, 14b/cb, 10b/cb, 23a/20b}
    \draw (\from) -- (\to);
    \foreach \from/\to in {11b/12b, 12b/13b, 13b/14b, 14b/10b, 10b/11b, 20b/21b, 21b/22b, 22b/23b, 23b/24b, 24b/20b}
    \draw (\from) -- (\to);
        \foreach \from/\to in {11b/22b, 12b/23b, 13b/24b, 14b/20b, 10b/21b, 21b/12b, 22b/13b, 23b/14b, 24b/10b, 20b/11b, 23b/203}
    \draw (\from) -- (\to);
    \end{tikzpicture}
\end{center}
\caption{The arm graph $\arm_3$}
\label{ARM3}
\end{figure}

We adopt the following notation. If there are $m$ zombies, we typically denote them as $z_1, z_2, \ldots, z_m$.  Let $d_i$ denote the distance from the survivor to $z_i$, a parameter that will typically change during the play of a game.  For integers $n$ and $r$, let $(n)_r$ denote the remainder $n$ (mod $r$); for example: $(7)_5 = 2$ and $(10)_5 = 0$.

\begin{lemma} \label{thm_one_arm}
Consider the graph obtained by combining $\arm_n$ with any graph \graph{H} by adding an edge from $\armout{2}{-1} \in V(\arm_n)$ to a vertex $v \in V(\graph{H})$. Suppose the survivor is at vertex \armout{0}{0}, and there are $m$ zombies $z_1, z_2, \ldots, z_m$, all on vertices of \graph{H}, such that $d_1 \le d_2 \le \cdots \le d_m$. Then the survivor has a winning strategy if
\[(d_2 - d_1)_5 + (d_3 - d_2)_5 + \cdots + (d_m - d_{m-1})_5 \le n-1.\]
\end{lemma}

\begin{proof}
\newcommand{\distone}{\ensuremath{\mathcal{Z}}}
\newcommand{\sgood}{-survivor-good}

We describe a winning strategy for the survivor. Let \distone \ refer to the set of zombies that are adjacent to the survivor when it is the survivor's turn.  The set \distone \ will change as the game goes on.
 
For $0 \le x \le n-1$, define a position to be $x$\dword{-survivor-good} if the following
conditions are satisfied:
\begin{enumerate}

\item The survivor is at vertex \armout{0}{x} and it is the survivor's turn to move.


\item The zombies not in \distone \ are in 
$\graph{H} \cup \{\armout{2}{-1}\} \cup \graph{T}_0 \cup \cdots \cup \graph{T}_{x-1}$.

\item
Every zombie in \distone \ is at one these vertices:
\armout{1}{x}, \armin{1}{x}, or \armout{2}{x-1}.

\item \label{cond_s_good}
$(d_2 - d_1)_5 + (d_3 - d_2)_5 + \cdots + (d_m - d_{m-1})_5 \le n-1-x$.

\end{enumerate}
The survivor strategy, depending on cases we will describe, will be to move from an $x$\sgood \ position to a new $x$\sgood \ position or to an $(x+1)$\sgood \ position. Either way we will guarantee that the sum $d_1 + \cdots + d_m$ strictly decreases until all the zombies are in \distone.

The survivor starts at \armout{0}{0} and should not move from \armout{0}{0} until a zombie arrives at the vertex \armout{2}{-1}, adjacent to the survivor. 
At this point, the zombies in \distone \ are all at vertex \armout{2}{-1}, so by the assumptions of the lemma the position is $0$\sgood.

We now describe what the survivor does in an $x$\sgood \ position.
The survivor is at vertex \armout{0}{x}.  Suppose
$\distone = \{ z_1, \ldots, z_i \}$, and if \distone \ does not include all the zombies, then
for zombies not in \distone, zombie $z_{i+1}$ is a zombie whose distance, $d_{i+1}$, to the survivor is minimum.  
From the vertex \armout{0}{x} the survivor's first
move will always be to \armout{4}{x} and thus all of the zombies in \distone \ will move to \armout{0}{x} or \armin{0}{x}.
Subsequently, there will be 2 different courses of move sequences, depending on  
whether the distance of $z_{i+1}$  from the survivor is more than 5 or not.

Case 1.  Suppose $d_{i+1} \ge 6$ or all the zombies are in \distone.  In this case, the survivor will follow a cycle around the outside vertices of $\graph{T}_x$, eventually arriving back at \armout{0}{x}, i.e. the survivor will move from \armout{4}{x} to \armout{3}{x} to \armout{2}{x} to
\armout{1}{x}, and finally return to \armout{0}{x}.  The zombies in \distone \ are forced to follow around, so that when the survivor returns to \armout{0}{x}, all the zombies who were in \distone \ will remain in \distone,  and be either at \armout{1}{x} or \armin{1}{x}.  Since the survivor made 5 moves in cycling around $\graph{T}_x$, any zombies that were not in \distone \ will be 5 steps closer to the survivor. Since $d_{i+1}$ was  initially at least 6, only the zombies initially in \distone \ can be in $\graph{T}_x$, though some other zombies may have joined \distone, by being at the vertex \armout{2}{x-1}.  Note the position is still $x$\sgood. 
Condition~\ref{cond_s_good} of being $x$\sgood \ holds because the sum on the left of the inequality has not changed:
$d_j$ is unchanged if $z_j$ started off in \distone, and otherwise $d_j$ is reduced by 5. 
If all the zombies were initially in \distone, then they 
still are, and otherwise the sum $d_1 + \cdots + d_m$ has strictly decreased.

Case 2.   Suppose $d_{i+1} \le 5$.  In this case, the survivor moves from
 \armout{4}{x} to \armout{3}{x} to \armout{2}{x}, and then
to \armout{0}{x+1} in $\graph{T}_{x+1}$.  The zombies that were in \distone \ remain in \distone, ending up at $\armout{2}{x}$. For a zombie that did not start out in \distone, its distance to the survivor was at least 2, so when the survivor moves to \armout{2}{x}, that zombie
is at \armout{0}{x}, or in $\graph{T}_y$ for $y < x$, or at \armout{2}{-1}, or in $H$.  Thus when the survivor moves from \armout{3}{x} to \armout{2}{x}, the distance to any of these zombies does not change, and the zombies not in \distone \ will get one step closer to the survivor. 
 Thus we note that
the sum $d_1 + \cdots + d_m$ has strictly decreased.  The position is now 
$(x+1)$\sgood. Most of the conditions are immediate.
We point out that Condition~\ref{cond_s_good} of being
$(x+1)$\sgood \ holds.
The value of
$x$ increased by 1, and so the right side of the inequality decreased by one.  As for the left side, it continues to be the case that 
$(d_2 - d_1)_5 = \cdots = (d_i - d_{i-1})_5 = 0$, but 
$(d_{i+1} - d_i)_5$ goes down by exactly 1, while 
$(d_{j+1} - d_j)_5$ is unchanged for $j \ge i+1$.  So the left side also decreases by one.

In both cases the sum $d_1 + \cdots + d_m$ is strictly reduced, so eventually the left side of the inequality in Condition~\ref{cond_s_good} must be reduced to zero (and is guaranteed to be reduced to zero should the survivor ever reach \armout{0}{n-1}). 
At that point the strategy will remain indefinitely in Case 1 with \distone \ eventually containing all the zombies and all zombies cycling around some copy of \graph{T}, one step behind the survivor.
\end{proof}

\begin{lemma} \label{1armUB}
Consider the graph obtained by combining $\arm_n$ with a graph \graph{H} by adding an edge from  $\armout{2}{-1} \in V(\arm_n)$ to a vertex $v \in V(\graph{H})$. Suppose the survivor is on a vertex in $\arm_n$ and there are zombies $z_1, \ldots, z_m$ such that $z_1$ is at $v$, and for all $1 \le i \le m-1$, $0 \le d_{i+1} - d_i \le 4$, and $d_m - d_1 \ge n$.
Then the zombies have a winning strategy.
\end{lemma}

\begin{proof}
\newcommand{\zgood}{-zombie-good}

We describe a winning strategy for the zombies. 
For $0 \le x \le n$, define a position to be $x$\dword{-zombie-good} if the following conditions hold:
\begin{enumerate}

\item The survivor is in $\graph{T}_y$, for $y \ge x$, or 
at the vertex \armout{0}{n}.

\item For some $i \ge 1$, all the zombies $z_1, \ldots, z_i$ are at
vertex \armout{0}{x}.

\item The zombies $z_{i+1}, \ldots, z_m$ are in
$\graph{H} \cup \{\armout{2}{-1}\} \cup \graph{T}_0 \cup \cdots \cup \graph{T}_{x-1}$.

\item
For $1 \le i \le m-1$, \ $0 \le d_{i+1} - d_i \le 4$.

\item \label{2zombies}
$d_m - d_1 \ge n-x$.

\end{enumerate}
The proof relies on the following claim, which we prove after giving the main argument.
\begin{claim}\label{maintech}
For $0 \le x < n$, if the game is in an $x$\zgood \ position, then the zombies can either catch the survivor in $\graph{T}_x$ or force an $(x+1)$\zgood \ position. 
\end{claim}

Throughout the analysis, we will assume the survivor does not move adjacent to a zombie unless forced to, 
since then the zombies immediately win and we are done. In the first two zombie moves, $z_1$ moves from $v$ to \armout{2}{-1} to \armout{0}{0}, with the other zombies following behind, and so for all $i$, $d_{i+1}-d_i$ remains unchanged, and the conditions of the lemma guarantee that this position is $0$\zgood.  By repeatedly applying Claim~\ref{maintech}, the zombies either catch the survivor, or force $x$\zgood \ positions for $x=1, 2, \ldots, n$.  In an $n$\zgood \ position both the survivor and $z_1$ are at \armout{0}{n}, so the survivor is caught.

Now we prove Claim~\ref{maintech}. Suppose the position is 
$x$\zgood, where $0 \le x < n$. 
First we dispense with the case where the survivor is in $\graph{T}_y$ where $y > x$ or at \armout{0}{n}.  In this case, $z_1$ can move to \armout{1}{x}, and since \armout{1}{x} is adjacent to  \armout{2}{x}, the survivor cannot move into $\graph{T}_x$ without being caught.  Thus on the next two moves, $z_1$ should move to \armout{2}{x}, and then 
\armout{0}{x+1}.  Since the relative distances between zombies do not change during these three moves, the position is now $(x+1)$\zgood.  

Thus we only need to consider the case in which
the position is $x$\zgood \ and the survivor is actually in
$\graph{T}_x$.  The zombies $z_1, \ldots, z_i$ all begin on vertex \armin{0}{x}, while the zombies $z_{i+1}, \ldots, z_m$ (Condition~\ref{2zombies} implies there must be at least one) begin in \graph{H} or at \armout{2}{-1} or in some $\graph{T}_y$ with $y < x$.  We will describe a strategy for $z_1$, which all the zombies $z_1, \ldots, z_i$ will follow, and a strategy for $z_{i+1}$.  It is not necessary to explicitly describe the strategy for $z_{i+2}, \ldots, z_m$, so long as they move toward the survivor on each turn. First 
we make an observation, which we will use repeatedly below.  
\begin{obs}
Suppose that $z_1$ is at $c^x$, $z_{i+1}$ is adjacent to \armout{0}{x} or \armout{2}{x-1}, and the survivor is at either \armin{0}{x} or \armin{4}{x}.
Then these two zombies can catch the survivor.
\end{obs}
We point out why the observation is true.
If the survivor moves to \armout{4}{x}, \armout{3}{x}, \armout{1}{x}or \armout{0}{x}, $z_1$ should move to 
\armin{4}{x}, \armin{3}{x}, \armin{1}{x}, or \armin{0}{x}, respectively.  In the first three cases, $z_1$ alone can catch the survivor on the next move.  In the fourth case (i.e. the survivor is at \armout{0}{x} and $z_1$ is at \armin{0}{x}), the zombie $z_{i+1}$ will have already caught the survivor before the survivor moves to \armout{0}{x} or will be at \armout{2}{x-1}.  Thus the survivor is already caught or will be caught on the zombies' next move.

Now we describe the strategies that $z_1$ and $z_{i+1}$ follow in order to either catch the survivor in $\graph{T}_x$ or force the position to be $(x+1)$\zgood.
The strategy for $z_{i+1}$ is simple: if adjacent to the survivor it catches the survivor; otherwise, it moves towards vertex 
\armout{0}{x} and once there, it moves to 
\armin{1}{x} or \armin{4}{x}, whichever is closer to the survivor.  If more moves are required of $z_{i+1}$, it just moves towards the survivor in any fashion. 
  The strategy for $z_1$ will depend on the survivor's position, so we consider two cases (we leave out vertices adjacent to
\armout{0}{x} since then $z_1$ immediately catches the survivor).

\textbf{Case 1:} Suppose the survivor is at $c^x$, \armin{2}{x} or \armout{2}{x}.
  Zombie $z_1$ should move to \armin{1}{x}, and then there are three subcases.
\begin{enumerate}[label={(\alph*)}]

\item If the survivor was at \armout{2}{x} and moves to \armout{0}{x+1}, then $z_1$ should move to \armout{2}{x} and then \armout{0}{x+1} for a total of three moves, which is a shortest path
from \armout{0}{x} to \armout{0}{x+1}. Note that now the survivor is in $\graph{T}_{x+1}$ or at \armout{0}{n} and no values of $d_i$ have changed. Thus the position is $(x+1)$\zgood.

\item \label{cx_to_4x}
If the survivor moves to \armin{4}{x}, $z_1$ should move to $c^x$.
 Then the zombies apply the observation to win.

\item If the survivor moves to \armin{3}{x} or \armout{3}{x}, then $z_1$ should move to \armin{2}{x}. This forces the survivor to move to \armin{4}{x} or \armout{4}{x}, and in either case $z_1$ should move to \armin{3}{x}. If the survivor is not already caught by $z_{i+1}$ and moves to \armout{0}{x} then the survivor will be caught by $z_{i+1}$.  Otherwise, the survivor must move to \armin{0}{x}.  In this case, $z_1$ should move to $c^x$ and 
the zombies apply the observation to win.

\end{enumerate}

\textbf{Case 2:} If the survivor is at \armin{3}{x} or \armout{3}{x}, then $z_1$ should move to \armin{4}{x}. This forces the survivor to move to \armin{2}{x} or \armout{2}{x}, 
and in either case $z_1$ should move to \armin{3}{x}. There are then two subcases.

\begin{enumerate}[label={(\alph*)}]

\item If the survivor is at \armout{2}{x} and moves to \armout{0}{x+1}, then $z_1$ should move to \armout{2}{x} and then to \armout{0}{x+1} for a total of 4 moves.  In this case, note that the zombies $z_{i+1}, \ldots, z_m$ are all one step closer to the survivor than before, because the survivor move from \armin{3}{x} or \armout{3}{x} to \armout{2}{x} did not increase the distance to any of these zombies.  To verify that Condition~\ref{2zombies} still holds, note that at the end of this sequence of moves, $d_{i+1}, \ldots, d_m$ have all decreased by 1, while the value of $x$ has increased by 1.  Thus the new position is $(x+1)$\zgood.

\item If the survivor moves to \armin{1}{x} or \armout{1}{x}, 
then $z_1$ should move to \armin{2}{x}. This forces the survivor to move to 
\armin{0}{x} or \armout{0}{x}. If the survivor moves to \armout{0}{x} then the survivor will be caught by $z_{i+1}$.  Otherwise, the survivor must move to \armin{0}{x}.  In this case, $z_1$ should move to $c^x$ and
the zombies apply the observation to win.

\end{enumerate}
\end{proof}

\section{A graph family exhibiting all feasible combinations of cop and zombie number}\label{zkmsec}

For integers $m \ge k \ge 1$, we will describe a graph with cop number $k$ and zombie number $m$, by first describing the base graph in Definition~\ref{def_basegraph}, and then describing how we hang arm graphs and paths off of the base graph.

\subsection{The base graph} \label{sec_bibd}

Let $V$ be a set of points, and $\mathcal{B}$ be a collection of subsets of $V$, called \dword{blocks}. Then the pair $(V,\B)$ is called a \dword{design}.  A design $(V,\B)$ is called a \dword{balanced incomplete block design} with parameters $(v,k,\lambda)$ if $|V|=v$, each block in $\B$ has cardinality $k$, and each pair of distinct points in $V$ is in exactly $\lambda$ blocks.

Recall that from any collection of sets, one can define its \dword{intersection graph}: The graph where the vertices are the sets and two sets are adjacent if their intersection is nonempty. Given any balanced incomplete block design $(V,\B)$, its \dword{block intersection graph} is the intersection graph of its blocks. Let $BIG(v,k,\lambda)$ denote the set of block intersection graphs of balanced incomplete block designs with parameters $(v, k,\lambda)$.  The cop numbers of graphs in $BIG(v,k,\lambda)$ are studied by Bonato and Burgess in Section 5 of \cite{BonBurg2017_Designs}.

\begin{lemmaNoNum} 
(Lemma 5.4 from \cite{BonBurg2017_Designs})
If $v > k(k-1)^2 + 1$, and $\graph{G} \in BIG(v,k,1)$ then $c(\graph{G}) \ge k$.
\end{lemmaNoNum}

Lemma 5.1 from \cite{BonBurg2017_Designs} states if  
$\graph{G} \in BIG(v,k,\lambda)$ then $c(\graph{G}) \le k$.  In the next lemma, we follow their proof; restricting our attention to $BIG(v,k,1)$, we show that $k$ zombies catch a survivor on these graphs.

\begin{lemma} \label{lemma_zombie_upper}
If 
$\graph{G} \in BIG(v,k,1)$, then $z(\graph{G}) \le k$.
\end{lemma}

\begin{proof}
Let $(V,\B)$ be a balanced incomplete block design with parameters $(v,k,1)$ and 
\graph{G} be its block intersection graph. We describe how $k$ zombies $z_1, \ldots, z_k$ can catch the survivor on \graph{G}.  Pick any $y \in V$ and have each zombie start at a block which contains $y$ (multiple zombies 
can start at the same vertex). Suppose the survivor starts at a block $B=\{x_1, \ldots, x_k\} \in \B$.  We can assume $y \notin B$ as otherwise the zombies could win on the next turn.
For $1 \le i \le k$, zombie $z_i$ should move to the unique block containing $y$ and $x_i$. Note that this move is a legal zombie move since it puts each zombie adjacent to the survivor.
Now wherever the survivor moves, it is to a block which intersects $B$, so at least one zombie is adjacent to the survivor and catches the survivor.
\end{proof}

\begin{defn} \label{def_basegraph}
For $k \ge 2$, let $\graph{ZC}_k$ be a specified element of $BIG(v,k,1)$ for some $v > k(k-1)^2 + 1$. 
Let $\graph{ZC}_1$ be a single
vertex.
\end{defn}

Wilson \cite{Wilson1972, Wilson1975} proved that for any positive integer $k$ there are infinitely many $v$ for which some balanced incomplete block design with parameters $(v, k, 1)$ exists, so the graph $\graph{ZC}_k$ exists for all $k$. Of course for a given $k$ there may be many possible choices for $v$, and many graphs in $BIG(v,k,1)$, so we have arbitrarily chosen one representative to be the graph $\graph{ZC}_k$. 
Since the zombie number is at least as large as the cop number for any graph, we use Lemma 5.4 from \cite{BonBurg2017_Designs} and
Lemma~\ref{lemma_zombie_upper}, in order to conclude that $c(\graph{ZC}_k) = z(\graph{ZC}_k) = k$.

\subsection{The final construction}\label{finalg}

We now describe the graph $\graph{Z}_{k,m}$.  For $n \ge 1$, denote by $P_n$ the path with $V(P_n) = \{v_1, \ldots, v_n\}$ and $E(P_n) = \{\edge{v_i}{v_{i+1}} \ \setbar \ 1 \le i \le n-1\}$.

\begin{defn}
Let $m \ge k \ge 1$.  Then $\graph{Z}_{k,m}$ is the graph obtained from $\graph{ZC}_k$ by adding the following for each vertex $v \in V(\graph{ZC}_k)$.

\begin{itemize}

\item
One copy of $\graph{P}_{4(m-k)}$ and an edge from $v$ to vertex $v_1$.

\item
$m$ copies of $\arm_{4(m-k)}$ and an edge from $v$ to the vertex \armout{2}{-1}
in each copy of $\arm_{4(m-k)}$.  

\end{itemize}

\end{defn}
The inclusion of the paths in $Z_{k,m}$ is primarily to provide an easy way to describe a winning starting position for $m$ zombies.

\begin{thm}
The graph $\graph{Z}_{k,m}$ has cop number $k$ and zombie number $m$.
\end{thm}

\begin{proof} 
We first show that $c(\graph{Z}_{m,k}) = k$. If a vertex $w$ 
is not in $\graph{ZC}_k$ (i.e. $w$ is on one of the attached arms or paths), then we refer to the closest vertex in $\graph{ZC}_k$ as the \dword{projection} of $w$.
Since $\graph{ZC}_k$ has cop number at least $k$, a robber can evade $k-1$ cops in $\graph{Z}_{k,m}$ simply by playing a winning strategy on $\graph{ZC}_k$, and regarding any cop on an arm or path as being at its projection.  Conversely, if $k$ cops play against the projection of the robber on $\graph{ZC}_k$, they can catch the projection.  Once this is done the robber is trapped on an arm or path.  On a path, one cop is sufficient to catch the robber simply by moving toward the robber.  On an arm, one cop can catch the robber as follows. Upon arriving at \armout{0}{x}, with the robber in $\graph{T}_y$ where $y\ge x$ or at \armout{0}{4(m-k)},
 the cop can move to \armin{0}{x} then to $c^x$.  After the robber's two responses, the only safe (i.e. non-adjacent) locations for the robber are vertices
\armout{1}{x}, \armout{2}{x}, \armout{3}{x}, \armout{4}{x}, \armout{0}{4(m-k)}, or some $\graph{T}_y$, for $y > x$.
If the robber is at \armout{4}{x}, \armout{3}{x}, or 
\armout{1}{x}, the cop moves to \armin{4}{x}, \armin{3}{x}, or \armin{1}{x}, respectively, thus
cornering the robber, and catching on the following move.  If the robber is either at \armout{2}{x}, \armout{0}{4(m-k)}, or in $\graph{T}_y$ where $y > x$, the cop should move to \armin{2}{x}, then \armout{2}{x}, and then \armout{0}{x+1}.  By following this strategy, the cop forces the robber down the arm, eventually catching the robber at \armout{0}{4(m-k)}.

For the zombie number, we first show that $m$ zombies can catch the survivor.  As in the proof of Lemma~\ref{lemma_zombie_upper}, start $k$ zombies in the graph $\graph{ZC}_k$ at a single vertex $y \in V(\graph{ZC}_k)$, calling one of 
them $z_1$.  
The remaining $m-k$ zombies $z_2, \ldots, z_{m-k+1}$ should start on vertices $v_4, v_8, \ldots, v_{4(m-k)}$ on the $P_{4(m-k)}$ that is attached to $y$.
Suppose the survivor starts at a vertex $v \in V(\graph{ZC}_k)$ or on a copy of 
$\arm_{4(m-k)}$ or $P_{4(m-k)}$ connected to $v$.  If $v$ is adjacent to $y$, all $k$ zombies on $y$ must move onto $v$.  If $v$ is at distance 2 from $y$, the $k$ zombies at $y$ play the strategy from Lemma~\ref{lemma_zombie_upper}, moving to mutual neighbors of $v$ and $y$ in $\graph{ZC}_k$.  Since this forces the survivor onto the path or one of the arms connected to $v$, on the next move all of these $k$ zombies will arrive at $v$. A third possible case is $v=y$.  In all three cases, after 1, 2, or 0 moves, respectively, we obtain a position where the survivor is on a path or arm connected to a vertex $v$, there are $k$ zombies at $v$, and none of the relative distances between the zombies have changed from their initial positions. If the survivor is on the path, then the zombies will win just by advancing down the path toward the survivor. 
Otherwise we can apply Lemma~\ref{1armUB} with with $4(m-k)$ in place of $n$ and $m-k+1$ in place of $m$. 
We have  the $m-k+1$ zombies $z_1, \ldots, z_{m-k+1}$
in the position required by Lemma~\ref{1armUB} such
that $d_{m-k+1} - d_1 \ge 4(m-k)$, so the zombies have a winning strategy.

We now show that the survivor can beat $m-1$ zombies using a strategy that has two phases. We say that a zombie in $\graph{Z}_{k,m}$ \textit{touches} $\graph{ZC}_{k}$ if it is in $V(\graph{ZC}_k)$ or adjacent to a vertex in $\graph{ZC}_k$.  If there are initially fewer than $k$ zombies touching $\graph{ZC}_k$, the survivor starts in Phase 1. Otherwise the survivor starts in Phase 2.  For Phase 1, the survivor plays the winning robber strategy on $\graph{ZC}_k$ described in \cite{BonBurg2017_Designs}, Lemma 5.4, by ignoring all cops who do not touch $\graph{ZC}_k$ and regarding those who are adjacent to $\graph{ZC}_k$ as being at their projections, and does this until at least $k$ zombies touch $\graph{ZC}_k$. Phase 1 is possible since $c(\graph{ZC}_k) = k$, which means that so long as fewer than $k$ zombies touch $\graph{ZC}_k$, the survivor cannot be caught.  Once $k$ or more zombies touch $\graph{ZC}_k$, the survivor moves to Phase 2.

As indicated above, the survivor may start in Phase 2, or transition to Phase 2 after starting in Phase 1.
If the survivor starts in Phase 2, then the survivor should start at vertex \armout{0}{0} on a copy of $\arm_{4(m-k)}$ which contains no zombies.  If the survivor transitions from Phase 1 to Phase 2, then the survivor should begin Phase 2 by moving to vertex \armout{2}{-1}, then to vertex \armout{0}{0},  on a copy of $\arm_{4(m-k)}$ which contains no zombies.  In either case, the survivor can move to a copy of $\arm_{4(m-k)}$ with no zombies since each vertex of $\graph{ZC}_k$ is attached to $m$ copies of $\arm_{4(m-k)}$, while there are only $m-1$ zombies.

Once in Phase 2, the survivor will remain in Phase 2, and win by applying Lemma~\ref{thm_one_arm}.
Since $\graph{ZC}_k$ has diameter 2, 
after the survivor moves to \armout{2}{-1} and then \armout{0}{0}, or simply starts at \armout{0}{0}, when it is the survivor's next turn to move, all the zombies who were initially touching $\graph{ZC}_k$ are now at distance 1, 2, 3, or 4 from the survivor.  Thus, supposing $d_1 \le d_2 \le \cdots \le d_{m-1}$, we conclude that 
\[(d_2 - d_1)_5 + (d_3 - d_2)_5 + \cdots + (d_k - d_{k-1})_5 \le 3.\]
Furthermore, we can conclude: 
\[\begin{split}
(d_2 - d_1)_5 + (d_3 - d_2)_5 + \cdots + (d_{m-1} - d_{m-2})_5 &\le 3+ (d_{k+1} - d_k)_5 + \cdots + (d_{m-1} - d_{m-2})_5\\
&\le 3+ 4(m-1-k)\\
&= 4(m-k)-1.
\end{split}\]
By  Lemma~\ref{thm_one_arm}, with $n=4(m-k)$, the survivor has a winning strategy.
\end{proof}

We remark that if one is not concerned with 1-cop-win graphs, but only cares about the above theorem when $k \ge 2$, then the construction and argument can be significantly simplified:
in place of the \graph{T} graph we could simply use a 
cycle $\graph{C}_5$.  The arguments basically work the same with this modification, and much of the case analysis of Lemma~\ref{1armUB} disappears.  However, the result for $k=1$ is one of the most significant concerns.

\section{Zombie number of the hypercube graph}\label{cube}

Let $\graph{Q}_n$ denote the $n$-dimensional hypercube.
In \cite{FHMP16}, Theorem 16, Fitzpatrick, Howell, Messinger, and Pike noted that $z(\graph{Q}_n) \ge \lceil\frac{2n}{3}\rceil$, which follows as a corollary from \cite{OO14}.
Conjecture~18 of \cite{FHMP16} states that $z(\graph{Q}_n) = \lceil 2n/3 \rceil$. We prove the conjecture is true. Independently, a more general result was just proven by Fitzgerald in \cite{FitzpatrickArxiv2018} which obtains the conjecture as a corollary.  Our proof is simpler, but just proves the conjecture.
Our proof is a modification of our proof of Theorem 3.1 from \cite{OO14}.  
\begin{thm}
For all $n \ge 1$, $z(\graph{Q}_n)= \lceil 2n/3\rceil$.
\end{thm}

\begin{proof}

In light of Theorem 16 from \cite{FHMP16}, it suffices to show that $\lceil 2n/3 \rceil$ zombies is enough to catch the survivor, by describing a winning zombie strategy.

We follow standard notation and write a vertex of $\graph{Q}_n$ as a binary vector of dimension $n$.  We will refer to the $n$ coordinates of the vector as coordinates $1, 2, \ldots, n$.  Since a move from one vertex to another changes (or \dword{flips}) one of the coordinates, we describe moves of the zombies and survivor in terms of flipping coordinates. For example, if $n=6$, a move from $(001101)$ to $(001001)$ is achieved by flipping the fourth coordinate.

If it is the survivor's turn, we say that a zombie is \dword{even} if its distance to the survivor is even, and otherwise we say that zombie is \dword{odd}. If it is the zombies' turn,
we say that a zombie is \dword{even} if its distance to the survivor is odd, and otherwise we say that zombie is \dword{odd}; i.e. on the zombies' turn we think of measuring a zombie's parity after it moves. 

We can now describe the zombie strategy.
We break the zombies into equal size groups or into groups whose sizes differ by at most one, 
putting one group (Group I) all on the single vertex $(0 \ldots 0 0)$ and the other (Group II) on
$(0 \ldots 0 1)$, more on the latter if need be. 
Since $\graph{Q}_n$ is bipartite, whatever the zombies or survivor do throughout the course of play, on either player's turn, all the zombies in Group I will have the same parity as one another (i.e. they are all even or all odd), and all the zombies in Group II will have the opposite parity.

On the zombies' first turn, and on every turn that occurs immediately after the survivor remains stationary, they will carry out
an organizational action we call a 
\dword{home setting}.
Suppose there are $e$ even zombies and $d$ odd zombies.
The $d$ odd zombies each choose a unique coordinate from among the coordinates $1, \ldots, d$ as their \dword{home coordinate}.  The $e$ even zombies choose unique coordinates among $d+2, d+4, \ldots, d+2e$ as their home coordinates.
Throughout the proof, we understand coordinates to \dword{wrap around}, that is,
for any coordinate $N>n$ we interpret this as the coordinate $k$ so that $k \equiv N \pmod{n}$.  A short calculation shows that the homes at least reach coordinate $n$ 
(i.e. $d+2e \ge n$), since $d + e = \lceil 2n/3 \rceil$ and $e \ge \lfloor (1/2) \lceil 2n/3 \rceil \rfloor$.
Every time the survivor chooses to remain stationary, all zombies have their parity flipped, i.e. all even zombies become odd, and all odd zombies become even.  Thus every time the survivor remains stationary, the zombies will carry out a new home setting.

From a zombie's home, say coordinate $k$, it regards the coordinates $k+1, k+2,  \ldots, n, 1, 2, \ldots k-1$, in that order, as being to the right of $k$. 
Define the \dword{reach} of the zombie with home at coordinate $k$ to be the number of consecutive coordinates on which the zombie's vector matches the survivor's, starting at coordinate $k$ and going right.  
Each zombie follows the same strategy relative to its home:
when it is the zombies' turn, each zombie simply flips the coordinate which makes its reach as large as possible.  For example, suppose $n = 7$ and the survivor is at vertex $(0000000)$.  Suppose a zombie is at vertex $(0110010)$ and its home is coordinate $4$.  Then this zombie's reach is currently $2$, and if it were the zombies' turn, this zombie would flip coordinate $6$ to make its reach $5$.
We will show that this strategy eventually catches the survivor.  

We say that a zombie has a \dword{critical reach} (only measured right before the survivor's turn) if its reach is $n-2$ for an even zombie, or $n-1$ for an odd zombie.
An even zombie with critical reach has two coordinates that differ from the survivor's; 
if the survivor flips one of these two coordinates, then that zombie flips the other to catch the survivor.
An odd zombie with critical reach has one coordinate that differs from the survivor's;
if the survivor flips that one, then the survivor has landed on the zombie and is caught.
Thus, when a zombie has critical reach, we say that the coordinate to the left of its home (or the two coordinates to the left of its home, in the case of an even zombie) are \dword{closed} to the survivor and assume that the survivor never flips such a coordinate.
 
Note that each time the survivor stays still, each zombie gets one step closer, so the survivor can only do this a finite number of times, and eventually we can assume the survivor always moves and so there is no home setting and no parity change for any zombie (recall that, except for the beginning of the game, there is only a home setting when the survivor remains stationary).
We consider the game from that point on, where the survivor never stays still.

We will show that eventually all the zombies will have critical reach, unless the survivor gets caught earlier.  
Once all the zombies have critical reach, since the homes at least reach coordinate $n$, all the coordinates will be closed to the survivor, so the survivor will lose no matter what is done.

To finish the proof and show that eventually all the zombies will have critical reach, it suffices to show that whatever the survivor does, every zombie's reach can at least be maintained, and some zombie without a critical reach can have its reach extended.  
Suppose the survivor flips coordinate $k$, which as mentioned above, we can assume is not a closed coordinate. 
Because of the positioning of the zombie homes and the fact that the zombie homes extend out to at least coordinate $n$, at least one of the coordinates, $k+1$ or $k+2$, is the home of some zombie (where coordinates may wrap around as mentioned above).
 If coordinate $k+1$ is the home of a zombie, then since coordinate $k$ was not yet closed to the survivor, that means the zombie with home $k+1$ did not yet have critical reach, so its reach can be extended. If coordinate $k+1$ is home to no zombie, but $k+2$ is the home of a zombie (in this case, it must be an even zombie), then since coordinate $k$ was not yet closed to the survivor, that means the zombie with home $k+2$ did not yet have critical reach, so its reach can be extended. In both cases, all other zombies can at least maintain their reach just by flipping the same coordinate the survivor last flipped.
\end{proof}

\end{document}